\documentclass[12pt]{article}
\usepackage{amssymb}
\usepackage{amsmath,amsthm}
\usepackage{hyperref}
\usepackage{enumerate}
\usepackage{color}
\usepackage{graphicx}
\usepackage{tkz-graph}
\usepackage{tikz}
\hypersetup{colorlinks=true, linkcolor=blue, citecolor=blue,urlcolor=blue}

 \newtheorem{remark}{Remark}

 \newtheorem{lemma}[remark]{Lemma}
 \newtheorem{theorem}[remark]{Theorem}
 \newtheorem{proposition}[remark]{Proposition}
 \newtheorem{corollary}[remark]{Corollary}

\title{Computing the metric dimension of a graph from primary subgraphs}

\author{D. Kuziak$^{(1)}$, J. A.
Rodr\'{\i}guez-Vel\'{a}zquez$^{(1)}$ and I. G. Yero$^{(2)}$
    \\
    \\
$^{(1)}${\small Departament d'Enginyeria Inform\`atica i Matem\`atiques,}\\
{\small Universitat Rovira i Virgili,}  {\small Av. Pa\"{\i}sos
Catalans 26, 43007 Tarragona, Spain.} \\{\small
dorota.kuziak\@@urv.cat, juanalberto.rodriguez\@@urv.cat}
\\
$^{(2)}${\small Departamento de Matem\'aticas, Escuela Polit\'ecnica Superior de Algeciras}\\
{\small Universidad de C\'adiz,} {\small
Av. Ram\'on Puyol s/n, 11202 Algeciras, Spain.} \\ {\small
ismael.gonzalez\@@uca.es}
\\
}

\date{}

\begin{document}

\maketitle

\begin{abstract}
Let $G$ be  a  connected graph. Given an ordered  set $W = \{w_1, w_2,\dots w_k\}\subseteq V(G)$   and a vertex $u\in V(G)$, the representation of $u$ with respect to $W$ is the ordered $k$-tuple $(d(u,w_1), d(u,w_2),\dots,$ $d(u,w_k))$, where $d(u,w_i)$ denotes the distance between $u$ and $w_i$. The set $W$ is a metric generator for $G$ if every two different vertices of $G$ have distinct representations. A minimum cardinality metric generator is called a \emph{metric
basis} of $G$ and its cardinality is called the \emph{metric dimension} of G. It is well known that the problem of finding the metric dimension of a graph is NP-Hard.  In this paper we obtain  closed formulae for the metric dimension of graphs with cut vertices. The main results are applied to
specific constructions including rooted product graphs, corona product graphs, block graphs and chains of graphs.
\end{abstract}



\section{Introduction}

Graph structures may be used to model computer networks. Servers, hosts or hubs in a network can be represented as vertices in a graph and edges could represent connections between them. Each vertex in a graph is a possible location for an intruder (fault in a computer network, spoiled device) and, this fact motivates the necessity of uniquely recognize each vertex of a graph, \emph{i.e.}, the possible location of an intruder in a network. This necessity gave rise to the notion of locating sets and locating number of graphs, introduced by Slater in \cite{Slater1975,Slater1988}. Harary and Melter \cite{Harary1976} also introduced independently the same concept, but using the terms resolving sets and metric dimension instead of locating sets and locating number, respectively. Moreover, in a more recent article, by Seb\"{o} and Tannier \cite{Sebo2004}, the terminology of metric generators and metric dimension for the concepts mentioned above, began to be used. In this article we follow the terminology and notation of  Seb\"{o} and Tannier \cite{Sebo2004}.

A {\em generator} of a metric space is a set $S$ of points in the space with the property that every point of the space is uniquely determined by its distances from the elements of $S$. Given a simple and connected graph $G$, we consider the metric $d_G:V(G)\times V(G)\rightarrow \mathbb{N}\cup\{0\}$, where $\mathbb{N}$ is the set of positive integers and $d_G(x,y)$ is the length of a shortest path between $x$ and $y$. The pair $(V(G),d_G)$ is readily seen to be a metric space. A vertex $v\in V(G)$ is said to distinguish two vertices $x$ and $y$ if $d_G(v,x)\ne d_G(v,y)$. A set $S\subset V(G)$ is said to be a \emph{metric generator} for $G$ if any pair of vertices of $G$ is distinguished by some element of $S$. A  metric generator $S$ is \textit{minimal}, if no proper subset $S'\subsetneq S$ is a metric generator for $G$. A minimal metric generator of minimum cardinality is called a \emph{metric basis} and its cardinality, the \emph{metric dimension} of $G$, is denoted by $\dim(G)$. Moreover, a minimal metric generator of maximum cardinality is called an \emph{upper metric basis} and its cardinality, the \emph{upper metric dimension} of $G$, is denoted by $\dim^+(G)$. For instance, for complete graphs of order $n$, $\dim^+(K_n)=\dim(K_n)=n-1$; for star graphs of order $r+1\ge 3$, $ \dim^+(K_{1,r})=\dim(K_{1,r})=r-1$; for cycle graphs of order $n$, $\dim^+(C_n)=\dim(C_n)=2$; and for path graphs of order $n\ge 3$, $\dim^+(P_n)=2>\dim(P_n)=1$.  The concepts of upper metric generator and upper metric dimension were introduced first in \cite{Chartrand2000a}.

On the other hand, studies about operations on graphs, particularly products of graphs, are being frequently presented and published in the last few decades. The metric dimension of Cartesian product graphs, lexicographic product graphs, strong product graphs, hierarchical product graphs and corona product graphs was studied in \cite{Caceres2007}, \cite{JanOmo2012,Saputro2013}, \cite{Rodriguez-Velazquez-et-al2014}, \cite{Feng2013} and \cite{Yero2011}, respectively. Furthermore, it was shown in \cite{Garey1979} that the problem of finding the metric dimension of a graph is NP-Hard. This suggests obtaining closed formulae for the metric dimension of special nontrivial families of graphs, or bounding the value of this invariant as tight as possible, or reducing the problem of computing the metric dimension of a graph to that of other simpler parameter. This last possibility regards the case of product graphs or, more general, those graphs obtained throughout some ``operations'' with other graphs, frequently called factor graphs or primary subgraphs.

Consider now a connected graph $G$ constructed from a family of pairwise disjoint (nontrivial) connected graphs $G_1,...,G_k$ in the following way. Select one vertex of $G_1$, one vertex of $G_2$, and identify these two vertices. Afterwards continue this procedure inductively. More precisely, let $G_1,...,G_i$ be already used in the construction, where $i\in\{2,...,k-1\}$. Select one vertex in the already constructed graph (particularly this vertex may be one of the already selected vertices) and one vertex of $G_{i+1}$, and then identify these two vertices. Figure \ref{example G} illustrated a geometrical representation of an example of a graph obtained in this manner.  The concept above was introduced in \cite{Deutsch-Klavzar2013}, where the authors used it to compute the Hosoya polynomials of a graph. Moreover, this construction was used in \cite{Rodriguez-Velazquez-Deutsh2014} to study the terminal Hosoya polynomial of composite graphs and in \cite{Rodriguez-Velazquez-et-al-2014} to compute the local metric dimension of graphs with cut vertices.

We say, as in \cite{Deutsch-Klavzar2013}, that $G$ is obtained by \textit{point-attaching} from $G_1,... ,G_k$ and that $G_i$'s are the \textit{primary subgraphs} of $G$. Furthermore, the vertices of $G$ obtained by identifying two vertices of different primary subgraphs are the \textit{attachment vertices} of $G$. We denote by $A(G)$ the set of attachment vertices of $G$ and by $A(G_i)$  the set of attachment vertices of $G$ belonging to  $V(G_i)$, \textit{i.e.}, $A(G_i)=A(G)\cap V(G_i)$. Observe that any graph constructed by point-attaching from a family of connected graphs has a tree-like structure, where the primary subgraphs are its building stones. Moreover,  for any $x,y\in V(G_i)$ it holds $d_G(x,y)=d_{G_i}(x,y)$.

Examples of graphs obtained by point-attaching are block graphs, cactus graphs, corona product graphs, rooted product graphs, bouquets of graphs,
circuits of graphs, chains of graphs,  etc.

\begin{figure}[h]
\centering
\begin{tikzpicture}
\draw (-6.1,0) ellipse (0.6cm and 1.5cm);
\node at (-6.1,0) {$G_1$};

\filldraw[fill opacity=1,fill=black] (-6.7,0) circle (0.09cm);
\node [left] at (-6.7,0) {\footnotesize $a$};
\draw (-7.7,0) ellipse (1cm and 0.4cm);
\node at (-7.7,0) {$G_2$};

\filldraw[fill opacity=1,fill=black] (-6.6,-1) circle (0.09cm);
\node [right] at (-6.6,-1) {\footnotesize $b$};
\draw[rotate=15] (-7.4,0.7) ellipse (0.8cm and 0.3cm);
\node at (-7.3,-1.2) {$G_3$};

\filldraw[fill opacity=1,fill=black] (-5.5,0) circle (0.09cm);
\node [right] at (-5.5,0) {\footnotesize $c$};
\draw (-4.5,0) ellipse (1cm and 0.7cm);
\node at (-4.5,0) {$G_4$};

\filldraw[fill opacity=1,fill=black] (-4.5,0.7) circle (0.09cm);
\node [above] at (-4.5,0.7) {\footnotesize $f$};
\draw (-4.5,1.7) ellipse (0.3cm and 1cm);
\node at (-4.5,1.7) {$G_7$};

\filldraw[fill opacity=1,fill=black] (-4.5,-0.7) circle (0.09cm);
\node [below] at (-4.5,-0.7) {\footnotesize $g$};
\draw (-4.5,-1.6) ellipse (0.4cm and 0.9cm);
\node at (-4.5,-1.6) {$G_8$};

\filldraw[fill opacity=1,fill=black] (-4.1,-1.6) circle (0.09cm);
\node [right] at (-4.1,-1.6) {\footnotesize $h$};
\draw (-3.1,-1.6) ellipse (1cm and 0.3cm);
\node at (-3.1,-1.6) {$G_{10}$};

\filldraw[fill opacity=1,fill=black] (-3.5,0) circle (0.09cm);
\node [right] at (-3.5,0) {\footnotesize $d$};
\draw (-2,0) ellipse (1.5cm and 0.5cm);
\node at (-2,0) {$G_5$};

\filldraw[fill opacity=1,fill=black] (-2,0.5) circle (0.09cm);
\node [above] at (-2,0.5) {\footnotesize $i$};
\draw (-2,1.4) ellipse (0.4cm and 0.9cm);
\node at (-2,1.4) {$G_{11}$};

\filldraw[fill opacity=1,fill=black] (-0.5,0) circle (0.09cm);
\node [left] at (-0.5,0) {\footnotesize $e$};
\draw[rotate=45] (0.7,0.4) ellipse (1cm and 0.3cm);
\node at (0.2,0.8) {$G_6$};
\draw[rotate=-45] (0.7,-0.4) ellipse (1cm and 0.3cm);
\node at (0.2,-0.7) {$G_9$};

\end{tikzpicture}
\caption{Sketch of a graph $G$ constructed by point-attaching from the primary subgraphs  $G_1,... ,G_{11}$.}
\label{example G}
\end{figure}
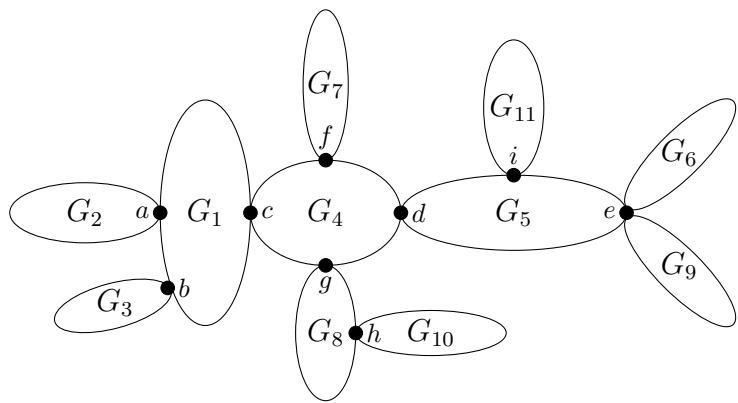

We say that a primary subgraph  $G_i$ is a \emph{primary end-subgraph} whenever   $|A(G_i)|=1$ and it is a   \emph{primary internal subgraph} whenever  $|A(G_i)|\ge 2$. For instance, $G_2$, $G_3$, $G_6$, $G_7$, $G_9$, $G_{10}$ and $G_{11}$ are primary end-subgraphs of the graph $G$ illustrated in Figure \ref{example G}, while $G_1$, $G_4$, $G_5$ and $G_8$ are primary internal subgraphs. In this case, $A(G_1)=\{a,b,c\}$, $A(G_2)=\{a\}$, $A(G_3)=\{b\}$ and so on.
Clearly, any graph obtained by point attaching contains at least two primary end-subgraphs.

In this paper we obtain  closed formulae for the metric dimension of graphs obtained by point-attaching. The main result is applied to
specific constructions including rooted product graphs, corona product graphs, block graphs and chain graphs.  To begin with, we need to introduce some additional notation and terminology. Given a simple graph $G$, the neighbourhood of a vertex $v\in V(G)$ is denoted by $N_G(v)$ and the eccentricity by  $\epsilon_G(v)$. The diameter  of $G$ is denoted $D(G)$, and given a set  $S\subset V(G)$, the subgraph of $G$ induced by $S$ is denoted by  $\langle S\rangle$.  A graph $G$ is $2$-\emph{antipodal} if for each vertex $x\in V(G)$ there exists exactly one vertex $y\in V(G)$ such that $d_G(x,y)=D(G)$. For example even cycles and hypercubes are $2$-antipodal graphs.
For the remainder of the paper, definitions will be introduced whenever a concept is needed.

\section{Main results}

We begin our exposition with a lower bound on the metric dimension of graphs from primary subgraphs in the general case. That is, when there is no rule for the construction of the graphs by point-attaching. Such constructions are of course depending on the attachment vertices of the primary subgraphs and, therefore, relatively complicate to deal with. In this sense, we shall use an extra parameter specifically related to the metric dimension of graphs from primary subgraphs, which we define below.

Let $G$ be a graph obtained by point-attaching  from $G_1,... ,G_k$.
An \emph{attaching metric generator} for a primary subgraph $G_i$ is a set $W\subseteq V(G_i)$ such that $W\cup A(G_i)$ is a metric generator for $G_i$.  A minimum cardinality attaching metric generator is called an \emph{attaching metric basis} and its cardinality, the \emph{attaching metric dimension} of $G_i$, is denoted by $\dim^*(G_i)$. For instance, assume that  $A(G_i)=\{v\}$. If   $v$ does not belong to any metric basis of $G_i$, then $\dim^*(G_i)=\dim(G_i)$ and if $v$ belongs to a metric basis of $G_i$, then
$\dim^*(G_i)=\dim(G_i)-1$. In particular,
for a  path graph or a  cycle graph of order $n$ we have that
$$\dim^*(P_n)=\left\{\begin{array}{ll}
                    1, & \mbox{if $P_n$ has exactly one attachment vertex which has degree $2$;} \\
                    0, & \mbox{otherwise.}
             \end{array} \right.
$$
$$\dim^*(C_n)=\left\{\begin{array}{ll}
                    1, & \mbox{if $C_n$ has exactly one attachment vertex or ($C_n$ has exactly two}\\
                       & \mbox{attachment vertices which are antipodal and $n$ is even);}\\
                    0, & \mbox{otherwise.}
             \end{array} \right.
$$
Furthermore,  for a complete graph we have that  $\dim^*(K_n)=n-|A(K_n)|-1$.

We are now able to state the following lower bound.
\begin{proposition}\label{lem low bound}
For any graph $G$ obtained by point-attaching from a family of connected graphs $G_1, . . . ,G_k$,
$$\dim(G)\ge \sum_{i=1}^k \dim^*(G_i).$$
\end{proposition}

\begin{proof}
Let $M$ be a metric basis of   $G$ and let $M_i=M\cap V(G_i)$, where $i\in \{1,...,k\}$. We claim that $M_i\cup A(G_i)$ is a metric generator for $G_i$. Let $u$ and $v$ be two different vertices of $G_i$. If $u$ and $v$ are not distinguished by any vertex in $M_i$, then they are distinguished by some vertex $y\in M_j$ for some $j\ne i$.
Let $x\in A(G_i)$ such that $d_G(y,x)=\displaystyle\min_{w\in V(G_i)}\{d_G(y,w)\}$. Hence,
$d_G(u,y)=d_G(u,x)+d_G(x,y)$ and $d_G(v,y)=d_G(v,x)+d_G(x,y)$. Since $d_G(u,y)\ne d_G(v,y)$, we have that $d_G(u,x)\ne d_G(v,x)$. So, $M_i\cup A(G_i)$ is a metric generator for $G_i$ and, as a consequence, $M_i$ is an attaching metric generator for $G_i$. Therefore, $|M_i|\ge \dim^*(G_i)$ and it follows that
$\dim(G)=|M|=\sum_{i=1}^k|M_i|\ge \sum_{i=1}^k \dim^*(G_i)$.
\end{proof}

%Once presented a general lower bound for the metric dimension of graphs from primary subgraphs, we continue with an upper bound for such a parameter. However, now we do not consider the general case, but we introduce some restrictions on the structure of the graphs obtained from primary subgraphs. Given a graph $G$ constructed by poit-attaching, we define the following properties of a primary subgraph $G_i$.

In order to show that the bond above is tight, we introduce some restrictions on the structure of the graphs obtained from primary subgraphs. Given a graph $G$ constructed by poit-attaching, we define the following properties of a primary subgraph $G_i$.
\\
\\
\noindent \textbf{Property} ${\cal P}_1$: For any $a\in A(G_i)$ and $z\in V(G_i)-A(G_i)$ there exists $b\in A(G_i)$ such that $d_{G_i}(a,b)\ge d_{G_i}(z,b)$.
\\
\\
\noindent \textbf{Property} ${\cal P}_2$:  $A(G_i)=\{v\}$ and either  $ G_i$  is not a path or  $G_i$ is a path  and $v$ is not a leaf.
\\
\\
Notice that property ${\cal P}_1$ is satisfied by a wide family of connected graphs. For instance, when a primary  internal subgraph $G_i$ holds one of the following conditions.
\begin{itemize}
  \item $A(G_i)=V(G_i).$
  \item $D(G_i)=2$ and $A(G_i)$ is any independent set for $G_i$.
  \item $\epsilon_{G_i}(x)=\epsilon_{G_i}(y)=d_{G_i}(x,y)$ for any pair of different vertices $x,y\in A(G_i)$. In particular, complete nontrivial graphs are included in this case.
  \item $G_i$ is $2$-antipodal and $A(G_i)$ is a set such that if $u\in A(G_i)$, then its antipodal vertex also belongs to $A(G_i)$.
\end{itemize}

It was shown in \cite{Chartrand2000} that $\dim(H)=1$ if and only if $H$ is a path. Also, $\{v\}$ is a metric basis of a path graph if and only if $v$ is a leaf.
Hence, if  $G_i$  satisfies ${\cal P}_2$, then  $\dim^*(G_i)\ge 1$ .

%To prove our next result we need to state the following one, which is already known from \cite{Chartrand2000}.

%\begin{theorem}{\em \cite{Chartrand2000}}\label{remark-path}
%$\dim(G)=1$ if and only if $G$ is a path.
%\end{theorem}

\begin{theorem}\label{th equal}
Let $G$ be a graph obtained by point-attaching from a family of connected graphs $G_1, . . . ,G_k$, $k\ge 3$, such that every  primary internal subgraph satisfies  ${\cal P}_1$, every primary end-subgraph satisfies  ${\cal P}_2$, and $A(G_i)\cap A(G_j)=\emptyset$ for any pair $G_i,G_j$ of primary end-subgraphs. Then
$$\dim(G)= \sum_{i=1}^k \dim^*(G_i).$$
\end{theorem}

\begin{proof}
By Proposition \ref{lem low bound}, $\dim(G)\ge \sum_{i=1}^k \dim^*(G_i).$ It remains to prove that  $\dim(G)\le \sum_{i=1}^k \dim^*(G_i).$

Let $S_i$ be an attaching metric basis of   $G_i$, $i\in \{1,...,k\}$. We shall show that $S=\bigcup_{i=1}^kS_i$ is a metric generator for $G$.
 To this end, we consider the following cases for two different vertices $x,y\in V(G)$.

\noindent Case 1. $x,y\in V(G_i)$. Since $S_i\cup A(G_i)$  is a metric generator for  $G_i$, there exists  $u\in S_i\cup A(G_i)$ such that $d_{G_i}(x,u)\ne d_{G_i}(y,u)$.  If $u\in S_i$, then we are done. Now, if $u\in A(G_i)$, then there exists a primary end-subgraph $G_j$, $j\ne i$, such that  for any $w\in S_j$, $d_G(u,w)=\displaystyle\min_{v\in V(G_i)}\{d_G(v,w)\}$. Notice that since  $G_j$ satisfies  ${\cal P}_2$, $S_j\ne \emptyset$.
Hence,
$$ d_G(x,w)=d_G(x,u)+d_G(u,w) \ne d_G(y,u)+d_G(u,w) =d_G(y,w).$$

\noindent Case 2. $x\in V(G_i)$ and $y\in V(G_j)$, where $i\ne j$. Let $a\in V(G_i)$ and $b\in V(G_j)$ be the attachment vertices such that $d_G(x,y)=d_G(x,a)+d_G(a,b)+d_G(b,y)$. Note that if $G_i$ and $G_j$ have a common attachment vertex, then $a=b$. If  $y=b=a$ or $x=a=b$, then we proceed as in Case 1, so we assume that $x$ and $y$ do not belong to the same primary subgraph, \textit{i.e.}, $y\ne a$ and $x\ne b$.
\\
\\
\noindent Subcase 2.1. $|A(G_i)|\ge 2$ or $|A(G_j)|\ge 2$. Without loss of generality, we assume  that $|A(G_i)|\ge 2$. Since $G_i$ satisfies ${\cal P}_1$, there exists $c\in A(G_i)-\{a\}$, such that  $d_{G_i}(a,c)\ge d_{G_i}(x,c)$.  Now, let $G_l$, $l\ne i$, be
 a primary end-subgraph  such that  for any $t\in S_l$, $d_G(c,t)=\displaystyle\min_{v\in V(G_i)}\{d_G(v,t)\}$  ($S_l\ne \emptyset$, as $G_l$ satisfies  ${\cal P}_2$). Then  for  any  $t\in S_l$,
$$
d_G(x,t)=d_G(x,c)+d_G(c,t) \le  d_G(a,c)+d_G(c,t)< d_G(y,a)+d_G(a,c)+d_G(c,t) =d_G(y,t).
$$
\noindent Subcase 2.2. $|A(G_i)|=|A(G_j)|=1$. Clearly $G_i$ and $G_j$ are primary end-subgraphs and since they satisfy  ${\cal P}_2$, it follows that $S_i$ and $S_j$ are not empty. Hence, let  $p\in S_i$ and $q\in S_j$. If $x,y$ are distinguished by $p$ or $q$, then we are done. On the contrary, suppose that neither $p$ nor $q$ distinguish the vertices $x$ and $y$. So, we have that
\begin{equation}
d_G(x,p)=d_G(y,p)=d_G(y,b)+d_G(b,a)+d_G(a,p)   \label{i1}
\end{equation}
and
\begin{equation}
d_G(y,q)=d_G(x,q)=d_G(x,a)+d_G(a,b)+d_G(b,q).  \label{i2}
\end{equation}
Observe that since  $A(G_i)\cap A(G_j)=\emptyset$, we have $a\ne b$. Moreover,
\begin{equation}
d_G(x,p)\le d_G(x,a) +d_G(a,p)  \label{i11}
\end{equation}
and
\begin{equation}
d_G(y,q)\le d_G(y,b)+d_G(b,q).  \label{i22}
\end{equation}
From (\ref{i1}) and (\ref{i11}) we obtain
\begin{equation}
d_G(y,b)+d_G(b,a)+d_G(a,p)\le d_G(x,a)+d_G(a,p),  \label{i111}
\end{equation}
and from (\ref{i2}) and (\ref{i22})
\begin{equation}
d_G(x,a)+d_G(a,b)+d_G(b,q)\le d_G(y,b)+d_G(b,q).  \label{i222}
\end{equation}
Finally, by adding (\ref{i111}) and (\ref{i222}) we have the following inequality
\begin{equation}
2\cdot d_G(a,b)\le 0,
\end{equation}
which is a contradiction.

According to the two cases above, $\dim(G)\le \sum_{i=1}^k \dim^*(G_i).$
\end{proof}

The next sections  are devoted to derive some consequences of Theorem \ref{th equal}. That is, we give closed formulae for the metric dimension of some specific families of graphs in terms of some parameters of its primary subgraphs, when the point-attaching process can be described as a graphs composition scheme or when the primary subgraphs satisfy some specific property.

\section{An extremal case}
  As above, let $G$ be a graph obtained by point-attaching  from $G_1,... ,G_k$. In this section we study the case where  every minimal metric generator for a primary subgraph is minimum \textit{i.e}.,  the case where  $\dim(G_i)=\dim^+(G_i)$.   Let ${\cal B}(G_i)$ is the set of metric bases of $G_i$ and let
$$\tau_i=\max_{B_j\in {\cal B}(G_i)} \left\{|A(G_i)\cap B_j| \right\}.$$
That is,  $\tau_i$ quantifies  the maximum number of attachment vertices of $G$ belonging simultaneously to a metric basis of   $G_i$.

\begin{corollary}\label{CorollaryExtremal}
Let $G$ be a graph obtained by point-attaching from a family of connected graphs $G_1, . . . ,G_k$, $k\ge 3$, such that   every  primary internal subgraph satisfies  ${\cal P}_1$, every primary end-subgraph satisfies  ${\cal P}_2$,  for any pair $G_i,G_j$ of primary end-subgraphs $A(G_i)\cap A(G_j)=\emptyset$ and $\dim(G_l)=\dim^+(G_l)$, whenever $A(G_l)\ne V(G_l)$. Then
$$\dim(G)=\sum_{i=1}^k (\dim(G_i)-\tau_i).$$
\end{corollary}

\begin{proof}
It is readily seen that for any primary subgraph $G_i$ of $G$ such that $\dim(G_i)=\dim^+(G_i)$, we have $\dim^*(G_i)=\dim(G_i)-\tau_i$. Therefore, the result is a direct consequence of Theorem \ref{th equal}.
\end{proof}

\begin{figure}[h]
\centering
\begin{tikzpicture}
%[line width=1pt,  scale=1]
\draw[black] (-3,0)--(-2,0)--(-1,0)--(0,-1)--(1,0)--(3,0)--(4,-1)--(4,1)--(3,0)--(5,0);
\draw[black] (1,0)--(2,1)--(3,0)--(1,0);
\draw[black] (-1,0)--(0,1)--(1,0);
\draw[black] (-2,-1)--(-2,0)--(-2,1);
\draw[black] (-1,-2)--(0,-1)--(1,-2)--(-1,-2);
\draw[black] (4,-1)--(5,0)--(4,1);
\draw[black] (-1,2)--(0,1)--(1,2)--(-1,2);

\coordinate [label={$a$}] () at (-1,0.1);
\coordinate [label={$c$}] () at (1,0.1);
\coordinate [label={$e$}] () at (3,0.1);
\coordinate [label={$b$}] () at (0,-0.9);
\coordinate [label={$d$}] () at (0,1.1);

\filldraw[fill=white]  (0,-1) circle (0.08cm);
\filldraw[fill=white]   (0,1) circle (0.08cm);
\filldraw[fill=white]   (-1,0) circle (0.08cm);
\filldraw[fill=white]   (1,0) circle (0.08cm);
\filldraw[fill=white]   (3,0) circle (0.08cm);
\filldraw[fill=black]  (2,1) circle (0.08cm);
\filldraw[fill=black]   (4,-1) circle (0.08cm);
\filldraw[fill=black]   (4,1) circle (0.08cm);
\filldraw[fill=black]   (-2,0) circle (0.08cm);
\filldraw[fill=black]   (-2,1) circle (0.08cm);
\filldraw[fill=black]   (-2,-1) circle (0.08cm);
\filldraw[fill=black]   (-3,0) circle (0.08cm);
\filldraw[fill=black]   (-1,-2) circle (0.08cm);
\filldraw[fill=black]   (1,-2) circle (0.08cm);
\filldraw[fill=black]   (-1,2) circle (0.08cm);
\filldraw[fill=black]   (1,2) circle (0.08cm);
\filldraw[fill=black]   (5,0) circle (0.08cm);
\end{tikzpicture}
\caption{A  graph $G$ obtained by point-attaching from  $G_1\cong K_{1,4}$, $G_2\cong C_{4}$, $G_3 \cong G_4\cong G_5\cong K_{3}$ and  $G_6\cong K_{4}$. In this case $A(G_1)=\{a\}$, $A(G_2)=\{a,b,c,d\}$, $A(G_3)=\{b\}$, $A(G_4)=\{c,e\}$, $A(G_5)=\{d\}$ and $A(G_6)=\{e\}$.}\label{FigureEctremal}
\end{figure}
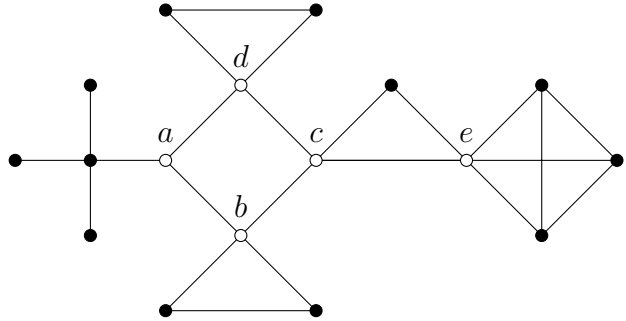

For the  graph $G$ shown in Figure \ref{FigureEctremal}  we have $\tau_1=1$, $\tau_2=2$, $\tau_3=1$, $\tau_4=2$, $\tau_5=1$ and $\tau_6=1$. In this case Corollary \ref{CorollaryExtremal} leads to $\dim(G)=6$.

A \emph{block graph} is a graph in which every biconnected component (block) is a clique. Note that any block graph is obtained by point-attaching from a family of complete graphs. For any complete graph of order $n$, $\dim(K_{n})=n-1=\dim^+(K_{n})$.  Then the following remark is a particular case of Corollary \ref{CorollaryExtremal}.

\begin{remark}
Let $G$ be a block graph obtained  from a family of complete graphs $\{K_{r_1},... ,K_{r_k}\}$, $k\ge 3$, such that any primary end-subgraph is different from $K_2$ and any two primary end-subgraphs have no common attachment vertex. Then
$$\dim(G)=\sum_{|A(K_{r_i})|<r_i} (r_i-|A(K_{r_i})|-1).$$
\end{remark}

\section{Rooted product graphs}

We continue in this section with an interesting particular case of graphs obtained by point-attaching: the rooted product of graphs. We must recall that some results on the metric dimension of rooted product graphs were already presented in \cite{Feng2013}. Nevertheless, several aspects on this topic were remaining from this work and also, the generalized version of rooted product graphs  was not studied. Next we give further results about that.

A \emph{rooted graph} is a graph in which one vertex is labeled in a special way so as to distinguish it from other vertices. The special vertex is called the \emph{root} of the graph. Let $G$ be a labeled graph on $n$ vertices and let ${\cal H}=\{H_1,... ,H_n\}$ be a family of rooted graphs. The \emph{rooted product graph} $G({\cal H})$ is the graph obtained by identifying the root of $H_i$ with the $i^{th}$ vertex of $G$ \cite{Godsil1978}. Clearly, any  rooted product  graph $G[{\cal H}]$  is a graph obtained by point-attaching from the primary internal subgraph  $G$, where $A(G)=V(G)$,  and the family ${\cal H}$  consists of  primary end-subgraphs having its attachment vertices in its roots.  From Theorem \ref{th equal} we deduce our next result.

\begin{corollary}\label{th general}
Let $G$ be a connected graph of order $n\ge 2$ and let ${\cal H}=\{H_1,... ,H_n\}$ be a family composed of  rooted graphs satisfying   ${\cal P}_2$, with roots $v_1,...,v_n$, respectively. Then
$$\dim(G({\cal H}))=\sum_{H_i\in \mathcal{H}_1} \dim(H_i)+\sum_{H_i\in \mathcal{H}_2} (\dim(H_i)-1),$$
where $H_i\in \mathcal{H}_1$ if  $v_i$ does not belong to any metric basis of   $H_i$ and $H_j\in \mathcal{H}_2$ if  $v_j$ belongs to a metric basis of   $H_j$.
\end{corollary}

We consider now the case of a family of vertex transitive graphs ${\cal H}$. Let $Aut(H)$ be the automorphism group of $H$. If $x,y\in V(H)$ and $\pi \in Aut(H)$, then $d(x,y)=d(\pi(x),\pi(y))$. So, if $S$ is a metric basis of   a connected graph $H$ and $\pi \in Aut(H)$, then $\pi(S)$ is a metric basis of   $H$. Thus, every vertex in a vertex transitive graph belongs to a metric basis and by using Corollary \ref{th general} we have the following.

\begin{remark}
Let ${\cal H}=\{H_1,... ,H_n\}$ be a family of   vertex transitive graphs of orders greater than two.
For any connected graph $G$ of order $n\ge 2$,
$$\dim(G({\cal H}))=\sum_{i=1}^n(\dim(H_i)-1).$$
In particular, if  ${\cal H}=\{K_{r_1},... ,K_{r_n}\}$, then $$\dim(G({\cal H}))=\sum_{i=1}^{n}(r_i-2)$$ and if ${\cal H}=\{C_{r_1},... ,C_{r_n}\}$, then $$\dim(G({\cal H}))=n.$$
\end{remark}

A particular case of rooted product graphs is when ${\cal H}$ consists of $n$ isomorphic rooted graphs \cite{Schwenk1974} (this was the case studied in \cite{Feng2013}). More formally, assuming that $V(G) = \{u_1, ..., u_n\}$ and that the root vertex of $H$ is $v$, we define the rooted product graph $G\circ_{v} H$, where $V(G\circ_{v} H)=V(G)\times V(H)$ and
$$E(G\circ_{v} H)=\displaystyle\bigcup_{i=1 }^n\{(u_i,b)(u_i,y): \; by\in E(H)\}\cup \{(u_i,v)(u_j,v):\; u_iu_j\in E(G)\}.$$

Figure \ref{ex rooted} shows two examples of rooted product graphs. We remark that this product was recently renamed as hierarchical product in \cite{Barriere2009}.

\begin{figure}[h]
\centering
\begin{tabular}{cccccc}
  %\hline
  % after \\: \hline or \cline{col1-col2} \cline{col3-col4} ...
\begin{tikzpicture}

\draw(1,1) -- (2,1) -- (3,1) -- (4,1);
\draw(0,2.5) -- (1,3.5) -- (1,1) -- cycle;
\draw(2.3,2) -- (1.5,3) -- (2,1) -- cycle;
\draw(2.7,2) -- (3.5,3) -- (3,1) -- cycle;
\draw(5,2.5) -- (4,3.5) -- (4,1) -- cycle;

\filldraw[fill opacity=0.9,fill=black]  (0,2.5) circle (0.08cm);
\filldraw[fill opacity=0.9,fill=black]  (1,1) circle (0.08cm);
\filldraw[fill opacity=0.9,fill=black]  (1,3.5) circle (0.08cm);
\filldraw[fill opacity=0.9,fill=black]  (1.5,3) circle (0.08cm);
\filldraw[fill opacity=0.9,fill=black]  (2,1) circle (0.08cm);
\filldraw[fill opacity=0.9,fill=black]  (2.3,2) circle (0.08cm);
\filldraw[fill opacity=0.9,fill=black]  (2.7,2) circle (0.08cm);
\filldraw[fill opacity=0.9,fill=black]  (3,1) circle (0.08cm);
\filldraw[fill opacity=0.9,fill=black]  (3.5,3) circle (0.08cm);
\filldraw[fill opacity=0.9,fill=black]  (4,1) circle (0.08cm);
\filldraw[fill opacity=0.9,fill=black]  (4,3.5) circle (0.08cm);
\filldraw[fill opacity=0.9,fill=black]  (5,2.5) circle (0.08cm);

\end{tikzpicture} & & \hspace*{0.7cm} & &
\begin{tikzpicture}
\draw(0,0.8) -- (4,0.8) -- (2,2) -- cycle;
\draw(0,0) -- (0,0.8) -- (0,1.6) -- (0,2.4);
\draw(2,1.2) -- (2,2) -- (2,2.8) -- (2,3.6);
\draw(4,0) -- (4,0.8) -- (4,1.6) -- (4,2.4);

\filldraw[fill opacity=0.9,fill=black]  (0,0) circle (0.08cm);
\filldraw[fill opacity=0.9,fill=black]  (0,0.8) circle (0.08cm);
\filldraw[fill opacity=0.9,fill=black]  (0,1.6) circle (0.08cm);
\filldraw[fill opacity=0.9,fill=black]  (0,2.4) circle (0.08cm);
\filldraw[fill opacity=0.9,fill=black]  (2,1.2) circle (0.08cm);
\filldraw[fill opacity=0.9,fill=black]  (2,2) circle (0.08cm);
\filldraw[fill opacity=0.9,fill=black]  (2,2.8) circle (0.08cm);
\filldraw[fill opacity=0.9,fill=black]  (2,3.6) circle (0.08cm);
\filldraw[fill opacity=0.9,fill=black]  (4,0) circle (0.08cm);
\filldraw[fill opacity=0.9,fill=black]  (4,0.8) circle (0.08cm);
\filldraw[fill opacity=0.9,fill=black]  (4,1.6) circle (0.08cm);
\filldraw[fill opacity=0.9,fill=black]  (4,2.4) circle (0.08cm);

\end{tikzpicture} \\
  %\hline
\end{tabular}
\caption{Rooted products $P_4\circ C_3$ and $C_3\circ_{v} P_4$, where $v$ has degree two.}
\label{ex rooted}
\end{figure}
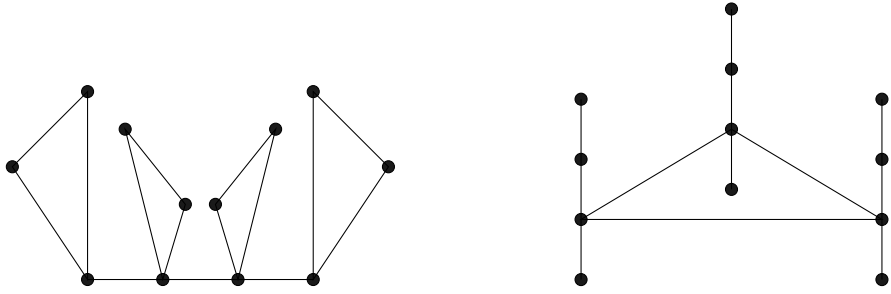

Notice that for the particular case of rooted product graphs $G\circ_v H$, Corollary \ref{th general} becomes the next propositions.

\begin{proposition}{\em \cite{Feng2013}}\label{pro 1}
Let $H$ be a connected graph and let $v$ be a vertex of $H$. If $v$ does not belong to any metric basis of   $H$, then for any connected graph $G$ of order $n$, $$\dim(G\circ_v H)=n\cdot \dim(H).$$
\end{proposition}

\begin{proposition}{\em \cite{Feng2013}}\label{pro 2}
Let $H$ be a connected graph different from a path and let $v$ be a vertex of $H$. If  $v$  belongs to a metric basis of   $H$, then for any connected graph $G$ of order $n\ge 2$, $$\dim(G\circ_v H)=n\cdot (\dim(H)-1).$$
\end{proposition}

Propositions \ref{pro 1} and \ref{pro 2} give rise to the problem of determining necessary and/or sufficient conditions for a vertex $v\in V(H)$ to belong to a metric basis of   $H$. For instance, it is easy to see that a vertex $v$ of a path $P$ belongs to a metric basis of   $P$ if and only if $v$ is a leaf of $P$. In connection with this fact, by using Proposition \ref{pro 1}, we have the following result.

\begin{corollary}\label{DimEqualto-n}
Let $H$ be a connected graph and let $v\in V(H)$ be a vertex not belonging to any metric basis of   $H$. For any connected graph $G$ of order $n$, $\dim(G\circ_v H)=n$ if and only if $H$ is a path graph and the root of $H$ is not a leaf.
\end{corollary}

We observe that in Proposition \ref{pro 1} the graph $H$ can be a path whenever the root is not a leaf. However in Proposition \ref{pro 2} paths are not allowed. This makes interesting the case of rooted product graphs $G\circ_v H$ when the graph $H$ is a path  and $v$ is a leaf. For that case, the following lower bound is known.
%The following lower bound for the metric dimension of $G\circ_v P$ where $v$ be a leaf of $P$ is known and was presented in \cite{Feng2013}.

\begin{proposition}{\em \cite{Feng2013}}\label{lem cota1}
Let $P$ be a path graph and let $v$ be a leaf of $P$. For any connected graph $G$ of order $n\ge 2$,
$$\dim(G\circ_v P)\ge \dim(G).$$
\end{proposition}

To obtain an upper bound we need some extra terminology and notation.  A \emph{dominating set} for a graph $G$ is a set $S\subseteq V(G)$ such that every vertex not in $S$ is adjacent to at least one member of $S$. The \emph{domination number} of $G$, denoted $\gamma(G)$, is the minimum cardinality of a dominating set. The following well-known upper bound on the domination number of a graph is useful to prove Lemma \ref{lem cota}.

\begin{theorem}{\rm (Ore, 1962)}\label{TheoremOre}
If a graph $G$ of order $n$ has no isolated vertices, then $$\gamma(G)\le \frac{n}{2}.$$
\end{theorem}

Given $X\subset V(G)$ we denote by $I(X)$ the set of isolated vertices of $\left\langle V(G)-X\right\rangle$. Also, for connected graphs we define  $I(G)=\displaystyle\max_{S\in {\cal B}(G)}\{|I(S)| \},$ where ${\cal B}(G)$ is the set of all the metric basis of   $G$.

\begin{proposition}\label{lem cota}
Let $P$ be a path graph and let $v$ be a leaf of $P$. For any connected graph $G$ of order $n\ge 2$,
$$\dim(G\circ_v P)\le \frac{\dim(G)+n-I(G)}{2}.$$
\end{proposition}

\begin{proof}
Let $S$ be a metric basis of   $G$ such that $I(G)=|I(S)|$. Let $S'$ be a dominating set for $\left\langle V(G)-(S\cup I(S))\right\rangle$. We show that $(S\cup S')\times \{v'\}$ is a metric generator for $G\circ_v P$, where $v'$ is the leaf of $P$ which is different from $v$. Let $(x,y),(x',y')$ be two different vertices of $G\circ_v P$. We differentiate the following cases.

\noindent Case 1. $y=y'$. In this case $x\ne x'$. So, there exists $u\in S$ such that $d_G(x,u)\ne d_G(x',u)$. If $u=x$ or $u=x'$, say $u=x$, then we clearly have that $d_{G\circ_v P}((x',y'),(u,v'))=d_{G\circ_v P}((x',y'),(x,y))+d_{G\circ_v P}((x,y),(u,v'))>d_{G\circ_v P}((u,v'),(x,y))$. Now, if $u\ne x$ and $u\ne x'$, then
\begin{align*}
d_{G\circ_v P}((u,v'),(x,y))&=d_{H}(v',v)+d_G(u,x)+ d_{H}(v,y)\\
                            &\ne d_{H}(v',v)+d_G(u,x')+ d_{H}(v,y)\\
                             &= d_{H}(v',v)+d_G(u,x')+ d_{H}(v,y')\\
                            &= d_{G\circ_v P}((u,v'),(x',y')).
\end{align*}

\noindent Case 2. $x=x'$. In this case $(a,v')$ resolves the pair $(x,y),(x',y')$, for every $a \in S$.

\noindent Case 3. $x\ne x'$ and $y\ne y'$. Since, the pair $(x,y),(x',y')$ is resolved by $(x,v')$ and also by $(x',v')$, we suppose  $x,x'\not\in S\cup S'$.
With this assumption in mind we consider the following subcases.

\noindent Case 3.1. $x,x'\not \in I(S)$.
In such a case, there exists $a,a'\in S'$ such that $x\in N_G(a)$ and $x'\in N_G(a')$. Now, if $x'\in N_G(a)$, then
\begin{align*}
d_{G\circ_v P}((a,v'),(x,y))&=d_{H}(v',v)+1+ d_{H}(v,y)\\
                            &\ne d_{H}(v',v)+1+ d_{H}(v,y')\\
                            &= d_{G\circ_v P}((a,v'),(x',y')).
\end{align*}
Analogously, if $x\in N_G(a')$, then we deduce that $(a',v')$ resolves the pair $(x,y),(x',y')$. Finally, we suppose that $x\not\in N_G(a')$ and $x'\not\in N_G(a)$. Since $d_G(a,x')\ge2$, $$d_{G\circ_v P}((a,v'),(x,y))=d_{H}(v',v)+1+ d_{H}(v,y)$$ and $$d_{G\circ_v P}((a,v'),(x',y'))=d_{H}(v',v)+d_G(a,x')+ d_{H}(v,y'),$$ we deduce that if $(a,v')$ does not resolve the pair $(x,y),(x',y')$, then $d_H(v,y')<d_H(v,y)$ and, as a consequence,
\begin{align*}
d_{G\circ_v P}((a',v'),(x',y'))&=d_{H}(v',v)+1+ d_{H}(v,y')\\
                            &< d_{H}(v',v)+1+ d_{H}(v,y)\\
                            &< d_{H}(v',v)+d_G(a',x)+ d_{H}(v,y)\\
                            &= d_{G\circ_v P}((a',v'),(x,y)).
\end{align*}

\noindent Case 3.2. $x \in I(S)$ and $x'\not \in I(S)$. In this case there exist $a\in S$ and $a'\in S'$ such that $x\in N_G(a)$ and $x'\in N_G(a')$. Now, if $x'\in N_G(a)$, then we proceed as in Subcase 3.1 and we obtain that $(a,v')$ resolves the pair $(x,y),(x',y')$. Analogously, if $x'\not\in N_G(a)$, then we obtain that either the pair $(x,y),(x',y')$ is resolved by $(a,v')$ or it is resolved by $(a',v')$.

\noindent Case 3.3. $x,x' \in I(S)$. In this case we take $a,a'\in S$ such that $x\in N_G(a)$ and $x'\in N_G(a')$ and we proceed as in Subcase 3.1.

Hence, $(S\cup S')\times \{v'\}$ is a metric generator for $G\circ_v P$. Moreover, by Theorem \ref{TheoremOre} we have $|S'|\le \frac{n-\dim(G)-|I(S)|}{2}$.
Therefore, $$\dim(G\circ_v P)\le |S|+|S'|\le \dim(G)+\frac{n-\dim(G)-|I(S)|}{2}= \frac{\dim(G)+n-I(G)}{2}.$$
\end{proof}

By Propositions \ref{lem cota1} and \ref{lem cota} we obtain the following.

\begin{proposition}\label{coro=dim(G)}
Let $G$ be a connected graph of order $n\ge 2$ and let $v$ be a leaf of a path graph $P$. If $I(G)=n-\dim(G)$, then $\dim(G\circ_v P)=\dim(G).$
\end{proposition}

The converse of Proposition \ref{coro=dim(G)} is false. For instance, $\dim(C_4\circ_v P)=\dim(C_4)=2$, while $I(C_4)=0$.

Note that $I(K_n)=n-\dim(K_n)=1$. Now we construct a family $\mathcal{F}$ of graphs where $I(G)=n-\dim(G)$, for every $G\in \mathcal{F}$. We begin with the star $S_{1,t}$, $t\ge 3$, of center $v$ and set of leaves $X=\{x_1,x_2,...,x_t\}$. Then to obtain a graph $G_t\in \mathcal{F}$ we add the set of vertices $Y=\{y_1,y_2,...,y_t\}$ and edges $x_iy_j$ for every $i,j\in \{1,...,t\}$ with $i\ne j$. Notice that for every $i,j\in \{1,...,t\}$, $i\ne j$, it follows $d(v,x_i)=1$, $d(v,y_i)=2$, $d(x_i,x_j)=2$, $d(y_i,y_j)=2$, $d(x_i,y_j)=1$ and $d(x_i,y_i)=3$. The graph $G_4$ is showed in Figure \ref{G_4}.

\begin{figure}[h]
\centering
\begin{tikzpicture}
\draw(-1.2,0.8) -- (1.2,-0.8);
\draw(-1.2,-0.8) -- (1.2,0.8);
\draw(1.2,0.8) -- (2.5,1.6);
\draw(1.2,-0.8) -- (2.5,-1.6);
\draw(-1.2,0.8) -- (-2.5,1.6);
\draw(-1.2,-0.8) -- (-2.5,-1.6);

\draw(2.5,1.6) -- (-1.2,0.8);
\draw(2.5,1.6) -- (1.2,-0.8);

\draw(2.5,-1.6) -- (1.2,0.8);
\draw(2.5,-1.6) -- (-1.2,-0.8);

\draw(-2.5,1.6) -- (1.2,0.8);
\draw(-2.5,1.6) -- (-1.2,-0.8);

\draw(-2.5,-1.6) -- (-1.2,0.8);
\draw(-2.5,-1.6) -- (1.2,-0.8);

\filldraw[fill opacity=0.9,fill=black]  (0,0) circle (0.08cm);
\filldraw[fill opacity=0.9,fill=black]  (1.2,0.8) circle (0.08cm);
\filldraw[fill opacity=0.9,fill=black]  (1.2,-0.8) circle (0.08cm);
\filldraw[fill opacity=0.9,fill=black]  (2.5,1.6) circle (0.08cm);
\filldraw[fill opacity=0.9,fill=black]  (2.5,-1.6) circle (0.08cm);
\filldraw[fill opacity=0.9,fill=black]  (-1.2,0.8) circle (0.08cm);
\filldraw[fill opacity=0.9,fill=black]  (-1.2,-0.8) circle (0.08cm);
\filldraw[fill opacity=0.9,fill=black]  (-2.5,1.6) circle (0.08cm);
\filldraw[fill opacity=0.9,fill=black]  (-2.5,-1.6) circle (0.08cm);

\node [below] at (0,0) {$v$};

\node [above] at (-1.2,0.8) {$x_1$};
\node [above] at (1.2,0.8) {$x_2$};
\node [below] at (1.2,-0.8) {$x_3$};
\node [below] at (-1.1,-0.8) {$x_4$};

\node [above right] at (2.5,-1.7) {$y_1$};
\node [above left] at (-2.5,-1.7) {$y_2$};
\node [below left] at (-2.5,1.7) {$y_3$};
\node [below right] at (2.5,1.7) {$y_4$};

\end{tikzpicture}
\caption{The graph $G_4$ satisfies $I(G_4)=n-\dim(G_4)=5$. The set $X=\{x_1,x_2,x_3,x_4\}$ is a metric basis of   $G_4$.}
\label{G_4}
\end{figure}
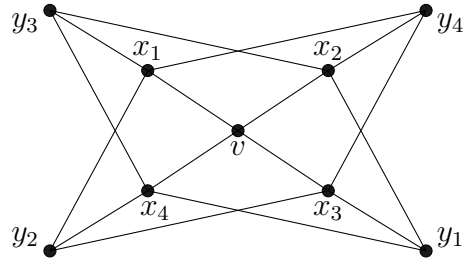

\begin{proposition}
For any graph $G\in {\cal F}$ of order $n$, $I(G)=n-\dim(G)$.
\end{proposition}

\begin{proof}
With the notation above we show that $X=\{x_1,x_2,...,x_t\}$ is a metric basis of   $G_t$. Since for every $i,j\in \{1,...,t\}$, $i\ne j$, $d(x_i,y_j)=1$, $d(x_i,y_i)=3$ and $d(v,x_i)=1$, we have that $X$ is a metric generator of $G_t$ and, as a consequence, $\dim(G_t)\le t$. Let $S$ be a set of vertices of $G_t$ such that $|S|<t$. We differentiate the following cases.

Case 1: $S\subsetneq X$. Let $x_j\notin S$. Since $d(x_l,v)=d(x_l, y_j)=1$ for  $l\ne j$, we have that $S$ is not a metric generator.

Case 2: $S\subsetneq Y$. Let $y_j\notin S$. Since $d(v,y_l)=d(y_l, y_j)=2$ for $l\ne j$, we have that $S$ is not a metric generator.

Case 3: $S\subsetneq X\cup \{v\}$ and $v\in S$. So, there exist at least two vertices $x_i,x_j\notin S$, $i\ne j$. Notice that $d(x_i, v)=d(x_j, v)=1$ and $d(x_i, x_l)=d(x_j, x_l)=2$ for every $l\ne i,j$. Thus, $S$ is not a metric generator.

Case 4: $S\subsetneq Y\cup \{v\}$ and $v\in S$. So, there exist at least two vertices $y_i,y_j\notin S$, $i\ne j$. Notice that $d(y_i, v)=d(y_j, v)=2$ and $d(y_i, y_l)=d(y_j, y_l)=2$ for every $l\ne i,j$. Thus, $S$ is not a metric generator.

Case 5: $S\cap X\ne \emptyset$, $S\cap Y\ne\emptyset$ and $v\notin S$. Since $|S|<t$, we can assume that there exists $y_j\notin S$ such that also $x_j\notin S$. Hence we have that $d(y_j,x_l)=d(v,x_l)=1$ for every $x_l\in S\cap X$ and $d(y_j,y_k)=d(v,y_k)=2$ for every $y_k\in S$. Thus, $S$ is not a metric generator.

Case 6: $S\cap X\ne \emptyset$, $S\cap Y\ne\emptyset$ and $v\in S$. Since $|S\cap (X\cup Y)|\le t-2$, there exist $i,j\in \{1,...,t\}$, $i\ne j$, such that   $x_i,x_j,y_i,y_j\notin S$. Notice that $d(x_i,v)=d(x_j,v)=1$, $d(x_i,x_l)=d(x_j,x_l)=2$ for every $x_l\in S$ and $d(x_i,y_k)=d(x_j,y_k)=1$ for every $y_k\in S$. Thus, $S$ is not a metric generator.

As a consequence of the cases above, we obtain that there is no metric generator of $G_t$ with cardinality less than $|X|$. Therefore $X$ is a metric base of $G_t$ and, as a consequence, $\dim(G_t)=t$. Finally, since $G_t$ has order $n=2t+1$ and the subgraph induced by $Y\cup \{v\}$ is empty, we obtain $I(G_t)=n-\dim(G_t)=t+1$.
\end{proof}

We continue observing the case when the roots of the paths in a rooted product graph $G\circ_v P$ are leaves, but now we consider when $G$ is a tree. A vertex of degree at least $3$ in a tree $T$ is called a \emph{major vertex} of $T$. Any leaf $u$ of $T$ is said to be a \emph{terminal vertex} of a major vertex $v$ of $T$ if $d_T(u, v)<d_T(u,w)$ for every other major vertex $w$ of $T$. The \emph{terminal degree} of a major vertex $v$ is the number of terminal vertices of $v$. A major vertex $v$ of $T$ is an \emph{exterior major vertex} of $T$ if it has positive terminal degree. Let $n_1(T)$ denotes the number of leaves of $T$, and let $ex(T)$ denotes the number of exterior major vertices of $T$. We can now state the formula for the dimension of a tree \cite{Chartrand2000}.

\begin{theorem}{\em \cite{Chartrand2000}}\label{value-loc-trees}
If $T$ is a tree that is not a path, then $$\dim(T) = n_1(T) - ex(T).$$
\end{theorem}

If $v$ is a leaf of a path $P$ and $T$ is a tree of order $n\ge 3$, then $T\circ_v P$ is a tree, $n_1(T\circ_v P)=n$ and $ex(T\circ_v P)=n-n_1(T)$. Hence, as a consequence of Theorem \ref{value-loc-trees} we deduce the following result.

\begin{corollary}\label{Arbol-camino}
Let $P$ be a path graph and let $v$ be a leaf of $P$. For any tree $T$ of order $n\ge 3$,
$$\dim(T\circ_v P)=  n_1(T).$$
\end{corollary}

The inequalities of Propositions \ref{lem cota1} and \ref{lem cota} lead to the following problem.   Given a path $P$ and a leaf $v$ of $P$, is there a graph $G$ of order $n$ such that $\dim(G)=a$ and $\dim(G\circ_v P)=b$, for every integers $a,b,n$ with $2\le a< b\le \frac{a+n}{2}$?

In order to give an answer to the question above, we construct a tree $T(a,b,n)$ in the following way. Let $S_{1,a}$ be a star graph with $a$ leaves and let $P'$ be a path graph of order $n-b+1$. To obtain $T(a,b,n)$ we proceed as follows.
\begin{itemize}
\item Identify one leaf of $P'$ with the center of the star $S_{1,a}$.
\item Add one pendant vertex to $b-a-1$ vertices of degree two of the path $P'$.
\end{itemize}
Since $P'$ has $n-b-1$ vertices of degree two, we have that $b-a-1\le n-b-1$. Thus, $b\le \frac{a+n}{2}$. Also, $n_1(T(a,b,n))=b$ and $ex(T(a,b,n)=b-a$. Thus, Theorem \ref{value-loc-trees} leads to $\dim(T(a,b,n))=a$ and, if $v$ is a leaf of a path graph  $P$, Corollary \ref{Arbol-camino} leads to $\dim(T(a,b,n)\circ_v P)=b$, which gives answer to the question mentioned above.
%A tree $T(3,7,12)$ is showed in Figure \ref{H(3,7,12)}.

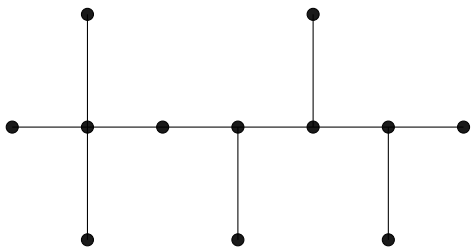
\begin{figure}[h]
\centering
\begin{tikzpicture}
\draw(0,0) -- (6,0);
\draw(1,0) -- (1,1.5);
\draw(1,0) -- (1,-1.5);
\draw(3,0) -- (3,-1.5);
\draw(4,0) -- (4,1.5);
\draw(5,0) -- (5,-1.5);

\filldraw[fill opacity=0.9,fill=black]  (0,0) circle (0.08cm);
\filldraw[fill opacity=0.9,fill=black]  (1,0) circle (0.08cm);
\filldraw[fill opacity=0.9,fill=black]  (2,0) circle (0.08cm);
\filldraw[fill opacity=0.9,fill=black]  (3,0) circle (0.08cm);
\filldraw[fill opacity=0.9,fill=black]  (4,0) circle (0.08cm);
\filldraw[fill opacity=0.9,fill=black]  (5,0) circle (0.08cm);
\filldraw[fill opacity=0.9,fill=black]  (6,0) circle (0.08cm);

\filldraw[fill opacity=0.9,fill=black]  (1,1.5) circle (0.08cm);
\filldraw[fill opacity=0.9,fill=black]  (1,-1.5) circle (0.08cm);
\filldraw[fill opacity=0.9,fill=black]  (3,-1.5) circle (0.08cm);
\filldraw[fill opacity=0.9,fill=black]  (4,1.5) circle (0.08cm);
\filldraw[fill opacity=0.9,fill=black]  (5,-1.5) circle (0.08cm);

\end{tikzpicture}
\caption{A tree $T(3,7,12)$.}
\label{H(3,7,12)}
\end{figure}

\begin{proposition}\label{rastreo}
Let $P$ be a path graph and let $v$ be a leaf of $P$. For any integer $a,b,n$ with $2\le a< b\le \frac{a+n}{2}$, there exists a graph $G$ of order $n$ such that $\dim(G)=a$ and $\dim(G\circ_v P)=b$.
\end{proposition}

\section{Corona product graphs}

We consider now an interesting construction, which can be understood as a rooted product graph and, consequently, as a graph obtained by using the point-attaching process.  The \emph{corona product graph} $G\odot {\cal H}$ is defined as the graph obtained from a graph $G$ of order $n$ and a family of graphs ${\cal H}=\{H_1,H_2,\ldots, H_n\}$ by adding an edge between each vertex of $H_i$ and the $i^{th}$-vertex of $G$, \cite{Frucht1970}. Hence, $G\odot {\cal H}$
is a rooted product graph $G(K_1+{\cal H})$ where $K_1+{\cal H}=\{K_1+H_1,K_1+H_2,...,K_1+H_n\}$ and $K_1+H_i$ is the join graph obtained from $K_1$ and $H_i$.
By Corollary \ref{th general} we deduce the following result.

\begin{remark}\label{MeinResultCorona}
Let $G$ be a connected graph of order $n\ge 2$ and let ${\cal H}=\{H_1,H_2,\ldots, H_n\}$ be a family of  nontrivial graphs. Then
$$\dim(G\odot {\cal H})=\sum_{H_i\in \mathcal{H}_1} \dim(K_1+H_i)+\sum_{H_i\in \mathcal{H}_2} (\dim(K_1+H_i)-1),$$
where $H_i\in \mathcal{H}_1$ if the vertex of $K_1$ does not belong to any metric basis of   $K_1+H_i$ and $H_j\in \mathcal{H}_2$  if the vertex of $K_1$ belongs to a metric basis of   $K_1+H_j$.
\end{remark}

The metric dimension of corona product graphs $G\odot {\cal H}$,  where ${\cal H}$ consists of $n$ graphs isomorphic   to a given graph $H$, was studied in  \cite{Rodriguez-Velazquez-Fernau2013,Fernau-Ja-Corona-2014,Iswadi2011,Yero2011}. In this case we use the notation $G\odot {  H}$ instead of $G\odot {\cal H}$.

We would emphasize the following particular case of the result above,  which improve some results obtained in \cite{Yero2011} and corrects a result\footnote{Corollary \ref{CorollaryResultCorrect} corrects Theorem 1 of \cite{Iswadi2011}, which states that if $H$ does not have dominating vertices, then $\dim(G\odot {H})=n\cdot \dim(K_1+H)$. A counterexample is shown in Figure \ref{FigureCounterexample}} stated in \cite{Iswadi2011}.

\begin{corollary}\label{CorollaryResultCorrect}
Let $G$ be a connected graph of order $n\ge 2$ and let $H$ be a  nontrivial graph.
 Then
$$\dim(G\odot {H})=
\left\{ \begin{array}{ll}
n\cdot ( \dim(K_1+H)-1),\; \mbox{\rm if the vertex of }\; K_1 \;\mbox{\rm belongs to  a metric}\\
\hspace{4.3cm} \mbox{ basis of } \; K_1+H;  \\
\\
n\cdot \dim(K_1+H), \; \mbox{\rm otherwise.}
\end{array} \right.
$$
\end{corollary}

\begin{figure}[h]
\centering
\begin{tabular}{cccccc}
  %\hline
  % after \\: \hline or \cline{col1-col2} \cline{col3-col4} ...
\begin{tikzpicture}
\draw(0,0) -- (0.8,0.8) -- (2,0.8) -- (2.8,0) -- (2,-0.8) -- (0.8,-0.8) -- cycle;
\draw(0.8,0.8)--(0.8,-0.8); \draw(2,0.8)--(2,-0.8);
\filldraw[fill opacity=0.9,fill=black]  (0,0) circle (0.08cm);
\filldraw[fill opacity=0.9,fill=white]  (0.8,0.8) circle (0.08cm);
\filldraw[fill opacity=0.9,fill=black]  (2,0.8) circle (0.08cm);
\filldraw[fill opacity=0.9,fill=black]  (2.8,0) circle (0.08cm);
\filldraw[fill opacity=0.9,fill=black]  (2,-0.8) circle (0.08cm);
\filldraw[fill opacity=0.9,fill=white]  (0.8,-0.8) circle (0.08cm);
\coordinate [label={$a$}] () at (0.8,0.81);
\coordinate [label={$b$}] () at (0.8,-1.4);
\end{tikzpicture}

  %\hline
\end{tabular}
\caption{A graph $H$ where $\dim(K_1+H)=3$. A metric basis of $K_1+H$ is $\{v,a,b\}$, where $v$ is the vertex of $K_1$.}
\label{FigureCounterexample}
\end{figure}
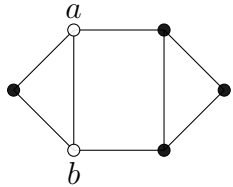

For instance,  for the graph $H$ shown in Figure \ref{FigureCounterexample} we have $\dim(K_1+H)=3$. A metric basis of $K_1+H$ is $\{v,a,b\}$, where $v$ is the vertex of $K_1$. Therefore, Corollary \ref{CorollaryResultCorrect} leads to  $\dim(G\odot {H})=2n$, for any graph $G$ of order $n\ge 2$.

Now, according to Remark \ref{MeinResultCorona}, a significant problem consists of determining necessary and/or sufficient conditions for the vertex of $K_1$ to belong to a metric basis of  $K_1+H$. For instance,  it was shown in \cite{Yero2011}
that if $H$ is a graph of diameter $D(H)\ge 6$ or it is a cycle graph of order greater than $6$, then the vertex of $K_1$ does not belong to any metric basis of $K_1+H$ and so $\dim(G\odot H)= n\cdot \dim(K_1+H)$. In this direction we state  the following result.

\begin{lemma}\label{Lemmak1belons}
Let $H$ be a graph of radius $r(H)$ and maximum degree $\Delta(H)$. If $r(H)\ge 4$ or $\dim(K_1+H)>\Delta(H)+1$, then the vertex of $K_1$ does not belong to any metric basis of   $K_1+H.$
\end{lemma}

\begin{proof}
Let $B$ be a metric basis of   $K_1+H$. We suppose that the vertex $v$ of $K_1$  belongs to $B$. Note that $v\in B$ if and only if there exists $u\in V(H)-B$ such that $B\subset N(u)$.

Now, if $r(H)\ge 4$, then we take  $u'\in V(H)$ such that $d_H(u,u')=4$   and  a shortest path $uu_1u_2u_3u'$. In such a case we have that $d_{K_1+H}(b,u_3)=d_{K_1+H}(b,u')=2$, for every $b\in B-\{v\}$, which is a contradiction. Hence, $v$ does not belong to any metric basis of   $K_1+H$. On the other hand, if $|B|>\Delta(H)+1$, then $v\not \in B$.
\end{proof}

The converse of Lemma \ref{Lemmak1belons} is not  true. In Figure \ref{K1+H} we show a graph $H$ of radius three where $\dim(K_1+H)=4<5=\Delta(H)+1$ and the vertex of $K_1$ does not belong to any metric basis of   $K_1+H$.

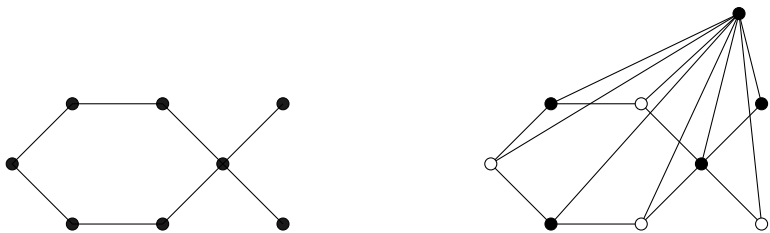
\begin{figure}[h]
\centering
\begin{tabular}{cccccc}
  %\hline
  % after \\: \hline or \cline{col1-col2} \cline{col3-col4} ...
\begin{tikzpicture}
\draw(0,0) -- (0.8,0.8) -- (2,0.8) -- (2.8,0) -- (2,-0.8) -- (0.8,-0.8) -- cycle;
\draw(2.8,0) -- (3.6,0.8);
\draw(2.8,0) -- (3.6,-0.8);

\filldraw[fill opacity=0.9,fill=black]  (0,0) circle (0.08cm);
\filldraw[fill opacity=0.9,fill=black]  (0.8,0.8) circle (0.08cm);
\filldraw[fill opacity=0.9,fill=black]  (2,0.8) circle (0.08cm);
\filldraw[fill opacity=0.9,fill=black]  (2.8,0) circle (0.08cm);
\filldraw[fill opacity=0.9,fill=black]  (2,-0.8) circle (0.08cm);
\filldraw[fill opacity=0.9,fill=black]  (0.8,-0.8) circle (0.08cm);
\filldraw[fill opacity=0.9,fill=black]  (3.6,0.8) circle (0.08cm);
\filldraw[fill opacity=0.9,fill=black]  (3.6,-0.8) circle (0.08cm);

\end{tikzpicture} & & \hspace*{0.9cm} & &
\begin{tikzpicture}
\draw(0,0) -- (0.8,0.8) -- (2,0.8) -- (2.8,0) -- (2,-0.8) -- (0.8,-0.8) -- cycle;
\draw(2.8,0) -- (3.6,0.8);
\draw(2.8,0) -- (3.6,-0.8);

\draw(3.3,2) -- (0,0);
\draw(3.3,2) -- (0.8,0.8);
\draw(3.3,2) -- (2,0.8);
\draw(3.3,2) -- (2.8,0);
\draw(3.3,2) -- (2,-0.8);
\draw(3.3,2) -- (0.8,-0.8);
\draw(3.3,2) -- (3.6,0.8);
\draw(3.3,2) -- (3.6,-0.8);

\filldraw[draw=black,fill=white]  (0,0) circle (0.08cm);
\filldraw[fill=black]  (0.8,0.8) circle (0.08cm);
\filldraw[draw=black,fill=white]  (2,0.8) circle (0.08cm);
\filldraw[fill=black]  (2.8,0) circle (0.08cm);
\filldraw[draw=black,fill=white]  (2,-0.8) circle (0.08cm);
\filldraw[fill=black]  (0.8,-0.8) circle (0.08cm);
\filldraw[fill=black]  (3.6,0.8) circle (0.08cm);
\filldraw[draw=black,fill=white]  (3.6,-0.8) circle (0.08cm);
\filldraw[fill=black]  (3.3,2) circle (0.08cm);

\end{tikzpicture} \\
  %\hline
\end{tabular}
\caption{A graph $H$ and the join graph $K_1+H$. White vertices form a metric basis of   $K_1+H$.}
\label{K1+H}
\end{figure}

 Remark \ref{MeinResultCorona} and Lemma \ref{Lemmak1belons}  lead to the next result.

\begin{proposition}
Let $G$ be a connected graph of order $n$ and let $H$ be a graph of radius $r(H)$ and maximum degree $\Delta(H)$. If $r(H)\ge 4$ or $\dim(K_1+H)>\Delta(H)+1$, then
$$ \dim(G\odot H)= n\cdot \dim(K_1+H).$$
\end{proposition}

It was shown in \cite{Yero2011} that any for corona graph $G=G_1\odot G_2$, such that $G_1$ is connected and both $G_1$ and $G_2$ are non-null graphs, it follows that the vertices of $G_1$ do not belong to any metric basis of $G$. Moreover, if $G_1$ has order $n_1\ge 2$ and  $H$ has diameter $D(G_1)\le 2$, then $\dim(G_1\odot G_2)=n_1\cdot \dim(G_2)$, \cite{Yero2011}.

%\begin{theorem}{\em \cite{Yero2011}}\label{theo-corona-k}
%Let $G_1$ be a connected graph of order $n_1\ge 2$ and let $H$ be a graph of diameter one or two, then
%$\dim(G_1\odot G_2)=n_1\cdot \dim(G_2)$
%\end{theorem}
Therefore, as a consequence of Corollary \ref{CorollaryResultCorrect} we obtain the following result.

\begin{remark}
Let $G$ and $G_1$ be two connected graphs of order $n$ and $n_1\ge 2$, respectively. Then for any $v\in V(G_1)$ and any graph $G_2$ of diameter one or two,
$$\dim(G\circ_v (G_1\odot G_2))=nn_1\cdot \dim(G_2).$$
\end{remark}

\section{Chain of graphs}

Let $G_1,G_2,... ,G_k$ be a finite sequence of pairwise disjoint (nontrivial) connected graphs and let $x_i,y_i\in V(G_i)$. A \emph{chain} $G$ is a graph obtained by point-attaching from $G_1,G_2,... ,G_k$ where the vertex $y_i$ is identified with the vertex $x_{i+1}$ for $i\in \{1,... ,k-1\}$.

\begin{figure}[h]
\centering
\begin{tikzpicture}
%[line width=1pt,  scale=1]
\draw[black] (-3,0)--(-2,0)--(-1,0)--(0,-1)--(1,0)--(3,0)--(4,-1)--(4,1)--(3,0);
\draw[black] (1,0)--(2,1)--(3,0)--(1,0);
\draw[black] (-1,0)--(0,1)--(1,0);
\draw[black] (-2,-1)--(-2,0)--(-2,1);

\coordinate [label={$a$}] () at (-1,0.1);
\coordinate [label={$b$}] () at (1,0.1);
\coordinate [label={$c$}] () at (3,0.1);

\filldraw[fill=black]  (0,-1) circle (0.08cm);
\filldraw[fill=black]   (0,1) circle (0.08cm);
\filldraw[fill=white]   (-1,0) circle (0.08cm);
\filldraw[fill=white]   (1,0) circle (0.08cm);
\filldraw[fill=white]   (3,0) circle (0.08cm);
\filldraw[fill=black]  (2,1) circle (0.08cm);
\filldraw[fill=black]   (4,-1) circle (0.08cm);
\filldraw[fill=black]   (4,1) circle (0.08cm);
\filldraw[fill=black]   (-2,0) circle (0.08cm);
\filldraw[fill=black]   (-2,1) circle (0.08cm);
\filldraw[fill=black]   (-2,-1) circle (0.08cm);
\filldraw[fill=black]   (-3,0) circle (0.08cm);
\end{tikzpicture}
\caption{A chain graph where $G_1\cong K_{1,4}$, $G_2\cong C_{4}$, $G_3\cong K_{3} \cong G_4$, $A(G_1)=\{a\}$, $A(G_2)=\{a,b\}$, $A(G_3)=\{b,c\}$ and  $A(G_4)=\{c\}$.}\label{FigureChain}
\end{figure}
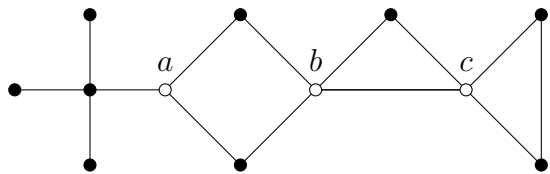

From Theorem \ref{th equal} we deduce our next result.

\begin{corollary}
Let $G$ be a chain obtained by point-attaching from a family of connected graphs $G_1,... ,G_k$, $k\ge 3$, such that $G_1$ and $G_k$ satisfy ${\cal P}_2$. If the attachment vertices of the primary subgraphs $G_i$ are diametral in $G_i$, for $i\in\{2,..., k-1\}$, then
$$\dim(G)=\sum_{i=1}^k \dim^*(G_i).$$
\end{corollary}

For instance, for the chain graph $G$ shown in Figure \ref{FigureChain} we have $\dim(G)=4$, as $\dim^*(G_1)=2$, $\dim^*(G_2)=1$, $\dim^*(G_3)=0$ and $\dim^*(G_4)=1$.

\end{document}